\documentclass[11pt]{amsart}
\usepackage{amsmath,amssymb}
\newtheorem{theorem}{Theorem}

\newtheorem{lemma}{Lemma}
\theoremstyle{definition}
\newtheorem{definition}{Definition}
\newtheorem*{remark*}{Remark}

\newtheorem*{acknowledgement*}{Acknowledgement}

\usepackage{latexsym,amssymb}
\def\Sr{S_{\mathfrak{z},2}(k;r)}

\begin{document}
\baselineskip=14.5pt
\title[The Determination of 2-color zero-sum generalized Schur Numbers]{The Determination of 2-color zero-sum\\ generalized Schur Numbers}

\author{Aaron Robertson}
\address[Aaron Robertson]{Department of Mathematics, Colgate University, Hamilton, New York}
\email[Aaron Robertson]{arobertson@colgate.edu}

\author{Bidisha Roy}
\author{ Subha Sarkar}
  \address[Bidisha Roy and Subha Sarkar]{Harish-Chandra Research Institute, HBNI, Jhunsi,
Allahabad, India}
\email[Bidisha Roy]{bidisharoy@hri.res.in}
\email[Subha Sarkar]{subhasarkar@hri.res.in}

\begin{abstract}
Consider the equation $\mathcal{E}: x_1+ \cdots+x_{k-1} =x_{k}$ and let $k$ and $r$ be positive integers such that $r\mid k$. The number $S_{\mathfrak{z},2}(k;r)$ is defined to be the least positive integer $t$ such that for any 2-coloring $\chi: [1, t] \to \{0, 1\}$ there exists a solution $(\hat{x}_1, \hat{x}_2, \ldots, \hat{x}_k)$ to the equation $\mathcal{E}$ satisfying 
$\displaystyle \sum_{i=1}^k\chi(\hat{x}_i) \equiv 0\pmod{r}$.  In a recent paper, the first author posed the question of determining the exact  value of $S_{\mathfrak{z}, 2}(k;4)$. In this article, we solve this problem and show,
more generally, that $S_{\mathfrak{z}, 2}(k, r)=kr - 2r+1$ for all positive integers $k$ and $r$ with $k>r$ and $r \mid k$.
\end{abstract}
\maketitle
\section{Introduction}

For $r \in  \mathbb{Z}^+$, there exists a least positive integer $S(r)$, called a {\it Schur number}, such that within every $r$-coloring of $[1, S(r)]$ there is a monochromatic solution to the linear equation $x_1+ x_2 =x_{3}$.

In 1933, Rado \cite{rado} generalized the work of Schur to arbitrary systems of linear equations. For any integer $k\geq 2$ and $r \in  \mathbb{Z}^+$, there exists a least positive integer $S(k;r)$, called a {\it generalized Schur number}, such that  every $r$-coloring of $[1, S(k;r)]$ admits a monochromatic solution of  equation $\mathcal{E} : x_1+ \cdots+x_{k-1} =x_{k}$.  Indeed, Rado \cite{rado} proved that the number $S(k,r)$ exists (is finite). In \cite{bb}, Beutelspacher and Brestovansky proved the exact value   $S(k;2) = k^2 - k - 1$.

Before we analogize the above number, we need the following definition.

\begin{definition}
Let $r \in  \mathbb{Z}^+$. We say that a set of integers $\{a_1, a_2, \ldots , a_n\}$ is {\it $r$-zero-sum}  if $\sum _{i=1}^n a_i   \equiv 0 \pmod r$.
\end{definition}

The Erd\H{o}s-Ginzburg-Ziv Theorem \cite{egz} is one of the cornerstones of   zero-sum theory (see, for instance, \cite{ad} and \cite{nd}). It states that {any sequence of $2n - 1$ integers must contain an $n$-zero-sum subsequence of $n$ integers}. In recent times, zero-sum theory has made remarkable progress (see, for instance, \cite{bd}, \cite{caro}, \cite{gg}, \cite{gbook}, \cite{gd}).

In \cite{aaron}, the first author replaced the ``monochromatic property'' of the generalized Schur number by the ``zero-sum property'' and introduced the following new number which is called a {\it zero-sum generalized Schur number}.

\begin{definition}
Let $k$ and $r$ be positive integers such that $r \mid k$. We define $S_{\mathfrak{z}}(k;r)$ to be the least positive integer $t$ such that for any $r$-coloring $\chi: [1, t] \to \{0, \ldots,r-1 \}$ there exists a solution 
$(\hat{x}_1, \hat{x}_2, \ldots, \hat{x}_k)$ to   $x_1+ \cdots+x_{k-1} =x_{k}$ satisfying $\displaystyle \sum_{i=1}^k\chi(\hat{x}_i) \equiv 0\pmod{r}$.
\end{definition} 

\noindent
{\bf Notation.} Throughout the article, we represent the equation $x_1+ \cdots+x_{k-1} =x_{k}$ by $\mathcal{E}$. 

Since $r \mid k$, note that if $(\hat{x}_1, \hat{x}_2, \ldots, \hat{x}_k)$ is a monochromatic solution to    equation $\mathcal{E}$,  then clearly it is an $r$-zero-sum solution. Hence, we get, 
$S_{\mathfrak{z}}(k;r) \leq S(k;r)$ and therefore, $S_{\mathfrak{z}}(k;r)$ is finite. 

In \cite{aaron},  the first author calculated lower   bounds of this number for some $r$. In particular, he proved the
following result. 
 \begin{theorem}\cite{aaron}\label{thm1} 
Let $k$  and $r$ be  positive integers such that $r \mid k$. Then,
$$
S_{\mathfrak{z}}(k;r) \geq
\begin{cases}
    3k-3   &  \text{ when }r=3;\\
    4k-4  &   \text{ when }r=4;\\
    2(k^2 - k - 1)  & \text{ when } r=k \text{ is an odd positive integer}.
\end{cases}
$$
\end{theorem}

In the same article, he introduced another number which is meant only for $2$-colorings; but keeping the $r$-zero-sum notion.

\begin{definition}
Let $k$ and $r$ be positive integers such that $r \mid k$. We denote by $S _{\mathfrak{z},2}(k;r)$ the least positive integer such that every 2-coloring of $\chi: [1,S _{\mathfrak{z},2}(k;r) ] \to \{0, 1 \}$ admits a solution $(\hat{x}_1, \hat{x}_2, \ldots, \hat{x}_k)$ to   equation $\mathcal{E}$ satisfying $\displaystyle \sum_{i=1}^k\chi(\hat{x}_i) \equiv 0\pmod{r}$.
\end{definition}

Since any $2$-coloring of $[1, S_{\mathfrak{z}}(k;r)]$ is also an $r$-coloring (for $r\geq 2$),  we see that $S_{\mathfrak{z},2}(k; r) \leq S_{\mathfrak{z}}(k;r)$
and hence $S_{\mathfrak{z},2}(k;r)$ is finite.

In \cite{aaron},  the first author proved the following theorem related to these 2-color zero-sum generalized
Schur numbers.

\begin{theorem}\cite{aaron}\label{thm2}
Let $k$  and $r$ be two positive integers such that $r \mid k$. Then,
$$
S_{\mathfrak{z},2}(k;r) =
\begin{cases}
    2k-3;   &  \text{if }r=2\\
    3k-5;  &   \text{if }r=3 \text{ and } k \neq 3\\
    k^2-k-1; &  \text{if }r=k
\end{cases}
$$
\end{theorem}
One notes that the exact values of $S_{\mathfrak{z},2}(k;r)$ for $r = 2,3$  and $S_{\mathfrak{z},2}(r,r)$ do not show any obvious generalization to $S_{\mathfrak{z},2}(k;r)$ for any $k$ which is a multiple of $r$. However, the computations given in \cite{aaron} when $r=4$ and $k =4, 8, 12$, and when $ r=5$ and $k=5, 10, 15$, were enough for us to conjecture a general formula,
which turns out to hold. To this end, by Theorem \ref{thm3} below,
we answer a question posed by the first author in \cite{aaron} and, more generally,
determine 
the exact values of $S_{\mathfrak{z},2}(k;r)$.

\begin{theorem} \label{thm3}
Let $k$ and  $r$ be positive integers  such that $ r \mid k$ and $k>r$. Then, 
$$S_{\mathfrak{z},2}(k;r)= rk -2r+1.
$$
\end{theorem}


\section{Preliminaries}
 
We start by presenting a pair of  lemmas useful for proving our upper bounds.

\begin{lemma}\label{lem1}
Let $k$ and $r$ be positive integers such that $r \mid k$ and $k\geq 2r$. Let $\chi: [1, rk -2r +1] \to \{0, 1\}$ be a 2-coloring such that $\chi(1) = \chi(r-1)=0$. Then there exists an $r$-zero-sum solution to  equation $\mathcal{E}$ under $\chi$.
\end{lemma}
\begin{proof}
Consider the solution $(1, 1, \ldots, 1, k-1)$.  If $\chi(k-1) = 0$, then, since $\chi(1) = 0$, we are done. Hence, we shall assume that $\chi(k-1) =1$. 

Next, we look at the solution 
$$ 
(\underbrace{1, \ldots, 1}_{k-r}, \underbrace{k-1,k-1,k-1}_{r-1}, rk-2r+1).$$
Since $\chi(1) = 0$ and $\chi(k-1) = 1$, we can assume that $\chi(rk-2r+1) = 0$; otherwise, we have
exactly $r$ integers of color $1$ and so the solution is $r$-zero-sum.

Since $(1, r, r, \ldots, r, rk-2r+1)$ is a solution to $\mathcal{E}$, we can assume that $\chi(r) = 1$. 

Finally, consider $$(\underbrace{r-1,\ldots, r-1}_{r-1},\underbrace{r,\ldots,r}_{k-r},rk-2r+1).$$ 
Since $\chi(r-1) = 0,  \chi(r) =1$, $\chi(rk-2r+1) = 0$, and $r \mid k$,  this solution is $r$-zero-sum, thereby
proving the lemma.
\end{proof}

\vskip 5pt
\begin{lemma}\label{lem2}
Let $k$ and $r$ be positive integers such that $r \mid k$ and $k\geq 2r$. Let $\chi: [1, rk -2r +1] \to \{0, 1\}$ be a coloring such that $\chi(1) = 0$ and $\chi(r-1) =1$. If one of the following holds, 
then there exists an $r$-zero-sum solution to  equation $\mathcal{E}$:

\begin{enumerate}
\item[\it(a)]$\chi(r) = 0$; 
 \item[\it(b)]$\chi(k-2) = 1$; 
 \item[\it(c)]$\chi(k-1) = 0$;
\item[\it(d)] $\chi(k) = 0$
\item[\it(e)]$\chi(rk-2r-1) = 0$;
\item[\it(f)]$\chi(rk-2r+1) = 1$.

\end{enumerate}
\end{lemma}

\begin{proof} We will prove each possibility separately; however, the order in which we do so matters so we will not be
proving them in the order listed.

\vskip 5pt
\noindent
(c) Consider $(1, 1, \ldots, 1, k-1)$.  If  $\chi(k-1) =0$ then this solution is $r$-zero sum.  

\vskip 5pt
\noindent
(d) By considering the solution $(r-1,\ldots,r-1,(r-1)(k-1))$, we can assume that $\chi((r-1)(k-1))=0$.
Using this in $(\underbrace{1,\ldots,1}_{k-r+1},\underbrace{k,\ldots, k}_{r-2},(r-1)(k-1))$ along
with the assumption that $\chi(k)=0$, we have an $r$-zero-sum solution.

\vskip 5pt
\noindent
(f) From part (c), we may assume that $\chi(k-1)=1$. Looking at  $(r-1, \ldots,r-1,k-1,rk-2r+1)$, since $\chi(k-1) =\chi(r-1) =1$, and we assume that $\chi(rk-2r+1)=1$, we have an $r$-zero sum solution.

\vskip 5pt
\noindent
(a) From part (f), we may assume that $\chi(rk-2r+1)=0$. With this assumption, we see that $( 1,r,\ldots, r, rk-2r+1)$ is $r$-zero-sum solution when $\chi(r)=0.$

\vskip 5pt
\noindent
(b) From parts (a) and (f), we may assume $\chi(r)=1$ and $\chi(rk-2r+1)=0$.  Under these
assumptions, we find that $(\underbrace{1, \ldots, 1}_{k-r-1}, r,\underbrace{k-2, \ldots, k-2}_{r-1}, rk-2r+1)$ is
an $r$-zero-sum solution with $\chi(k-2)=1$.

\vskip 5pt
\noindent
(e) By considering $(\underbrace{1, \ldots, 1}_{r-1}, \underbrace{r,\dots,r}_{k-2r},\underbrace{2r-3, \ldots, 2r-3}_{r}, rk-2r-1)$ and using $r \mid k$, we have an $r$-zero-sum solution only when $\chi(rk-2r-1)=0$.
\end{proof}

\section{Proof of the Main Result}

First, we restate the main result.

\vskip 5pt
\noindent
{\bf Theorem 3.} 
Let $k$ and  $r$ be positive integers  such that $ r \mid k$ and $k\geq 2r$. Then, 
$$S_{\mathfrak{z},2}(k;r)= rk -2r+1.
$$

\vskip 5pt
\noindent
{\it Proof.} 
We start with the lower bound.
To prove that $S_{\mathfrak{z},2}(k;r) > rk -2r $, we consider the following $2$-coloring $\chi$ of $[1, rk-2r]$  defined by $(0)^{k-2} (1)^{rk-k-2r+2}$.  Assume, for a contradiction, that $\chi$ admits an $r$-zero-sum
solution $(\hat{x}_1, \hat{x}_2, \ldots, \hat{x}_k)$ to equation $\mathcal{E}$. Then $\chi(\hat{x}_i) = 1$ for some $i \in \{1,2,\dots,k\}$; otherwise the solution is monochromatic of color $0$,
but $\sum_{i=1}^{k-1}\hat{x}_i \geq k-1$, meaning that $\hat{x}_k$ cannot be of color 0. 

Assuming that $(\hat{x}_1, \hat{x}_2, \ldots, \hat{x}_k)$ is  $r$-zero-sum and not monochromatic of color 0,  we must have $\chi(x_j) = 1$ for  at least $r$ of the $x_j$'s.  Since the minimum integer under $\chi$ that is of color $1$ is $k-1$, this gives us 
$$
\sum_{i=1}^{k-1} x_i \geq (r-1) (k-1) + 1(k-r)  = (rk-2r+1)> rk-2r,
$$
which is out of bounds, a contradiction. Hence, $\chi$ does not admit an
$r$-zero-sum solution to $\mathcal{E}$ and we
conclude that $\Sr \geq rk-2r+1$.

We now move on to the upper bound.
We let $\chi:[1,rk-2r+1]\rightarrow \{0,1\}$ be an arbitrary $2$-coloring. We may assume  that $\chi(1) = 0$, since $\chi$ admits an $r$-zero-sum solution if and only if the induced coloring $\overline{\chi}$ defined by $\overline{\chi}(i) = 1- \chi(i)$ also does so.

The cases $r=2,3$ have been done by Theorem \ref{thm2}.  Hence,
we may assume that $r \geq 4$.  We must handle the case $r=4$ separately; we start with this case.

We will show that $4k-7$ serves as an upper bound for $S_{\mathfrak{z}}(k;r)$.
Consider the following solution to $\mathcal{E}$:
$$
 (1,1,1, \underbrace{2, \ldots, 2}_{k-8}, 3,3, k, k, 4k-7).
$$
 
Noting that $r-1=3$ and $rk-2r+1=4k-7$, by Lemmas \ref{lem1} and \ref{lem2}, we may assume $\chi(3)=1$, $\chi(k)=1$, 
and $\chi(4k-7)=0$.  Since $k$ is a multiple of $4$ and $k \geq 8$, we see that $k-8$ is also a multiple of
4.  Hence, the color of $2$ does not affect whether or not this solution is $r$-zero-sum.  Of the
integers not equal to 2, we have exactly four of them of color 1.  Hence, this solution is
$4$-zero-sum.  This, along with the lower bound above, proves that $S_{\mathfrak{z},2}(k;4)=4k-7$.

We now move on to the cases where $r \geq 5$.  We proceed by  assuming that no $r$-zero-sum solution occurs under an arbitrary 2-coloring $\chi:[1,rk-2r+1] \rightarrow \{0,1\}$.   From Lemmas \ref{lem1} and \ref{lem2}, we may assume
the following table of colors holds.

$$
\begin{array}{l|l}
\mbox{color 0}&\mbox{color 1}\\\hline
1&r-1\\
k-2&r\\
rk-2r+1&k-1\\
&k\\
&rk-2r-1.
\end{array}
$$

In order for the solution
$$
(\underbrace{1, \ldots, 1}_{k-r}, \underbrace{k-2, \ldots, k-2}_{r-2},k+r-3, rk-2r+1)
$$ 
not to be $r$-zero-sum, we deduce that $\chi(k+r-3)=1$.  Using this in the solution
$$
(\underbrace{1, \ldots, 1}_{k-r}, \underbrace{k, \ldots, k}_{r-2}, k+r-3,  rk-3)
$$ 
we may assume that $\chi(rk-3)=0$.  In turn, we use this in
$$
(\underbrace{2, \ldots, 2}_{k-r-1}, r,r-1,\underbrace{k, \ldots, k}_{r-2},  rk-3)
$$ 
to deduce that $\chi(2)=1$.  Modifying this last solution slightly, we consider
$$
(\underbrace{3, \ldots, 3}_{k-r-1}, r,r,r,\underbrace{k, \ldots, k}_{r-3},  rk-3)
$$ 
to deduce that $\chi(3)=1$.  Finally, since $r \geq 5$, we can consider
$$
(\underbrace{2, \ldots, 2}_{k-2r+6}, \underbrace{3, \ldots, 3}_{r-5}, \underbrace{k-1, \ldots, k-1}_{r-2},  rk-2r-1).
$$ 
We see that this solution is monochromatic (of color 1), and, hence, is $r$-zero-sum.
This proves that $\Sr \leq rk-2r+1$ for $r \geq 5$, which, together with the lower bound at the
beginning of the proof, gives us $\Sr=rk-2r+1$, thereby completing the proof.
\hfill $\Box$

\vskip 20pt
\begin{acknowledgement*} 
We would  like to sincerely thank Prof. R Thangadurai for his insightful remarks. This work was done while second and third authors were visiting the Department of Mathematics, Ramakrishna Mission Vivekananda Educational and Research Institute and they wish to thank this institute for the excellent environment and for their  hospitality.
\end{acknowledgement*}

\end{document}